\documentclass[letterpaper,conference]{ieeeconf}
\IEEEoverridecommandlockouts                              % This command is only
\usepackage[utf8]{inputenc}
\usepackage{textcomp}

\usepackage{graphicx}
\usepackage{amsmath}
\usepackage[version=4]{mhchem}
\usepackage{siunitx}
\usepackage{amsthm}
\usepackage{graphicx}

\DeclareMathOperator{\sech}{sech}
\usepackage{longtable,tabularx}
\usepackage[style=ieee ,sorting=none,maxbibnames=99,giveninits, doi=false]{biblatex}
\usepackage{amsmath} % assumes amsmath package installed
\usepackage{amsfonts}

\newtheorem{theorem}{Theorem}
\newtheorem{assumption}{Assumption}
\newtheorem{proposition}{Proposition}[section]
\newtheorem{remark}{Remark}

\usepackage{float} 
\usepackage{tikz}
\usetikzlibrary{shapes,arrows,calc,positioning}
\tikzstyle{bigblock} = [draw, fill=blue!20, rectangle, 
    minimum height=6em, minimum width=8em]
\tikzstyle{medblock} = [draw, fill=blue!20, rectangle, 
    minimum height=4em, minimum width=4em]    
\tikzstyle{mux} = [draw, fill=black!20, rectangle, 
    minimum height=5em, minimum width=0.1em]    
\tikzstyle{smallblock} = [draw, fill=blue!20, rectangle, 
    minimum height=2em, minimum width=3em]
    
\tikzstyle{data_block} = [draw, fill=green!20, rectangle, 
    minimum height=2em, minimum width=3em]
\tikzstyle{ops_block} = [draw, fill=blue!20, rectangle, 
    minimum height=2em, minimum width=3em]    
\tikzstyle{est_block} = [draw, fill=red!20, rectangle, 
    minimum height=2em, minimum width=3em]    
    
\tikzstyle{sum} = [draw, fill=blue!20, circle, node distance=1cm,minimum height=0.5cm]
\tikzstyle{signal} = [coordinate]
\tikzstyle{pinstyle} = [pin edge={to-,thin,black}]
\tikzstyle{block} = [draw, fill=blue!20, rectangle, 
    minimum height=3em, minimum width=9em]
\tikzstyle{blockS} = [draw, fill=blue!20, rectangle, 
    minimum height=3em, minimum width=4em]    
\tikzstyle{input} = [coordinate]
\tikzstyle{output} = [coordinate]
\usetikzlibrary{matrix}

\usetikzlibrary{positioning}
\usetikzlibrary{math}

\usepackage{hyperref}
\hypersetup{
    colorlinks,
    linkcolor={blue!100!black},
    citecolor={blue!50!black},
    urlcolor={blue!80!black}
}

\newcommand{\bc}{\begin{center}}
\newcommand{\ec}{\end{center}}
\newcommand{\benum}{\begin{enumerate}}
\newcommand{\eenum}{\end{enumerate}}
\newcommand{\nn}{\nonumber}
\newcommand{\matl}{\left[ \begin{array}}
\newcommand{\matr}{\end{array} \right]}

\renewcommand{\matl}{\begin{bmatrix}}
\renewcommand{\matr}{\end{bmatrix}}

\newcommand{\matls}{\left[ \begin{smallmatrix}}
\newcommand{\matrs}{\end{smallmatrix} \right]}
\newcommand{\isdef}{\stackrel{\triangle}{=}}

\newcommand{\inv}{^{-1}}

\newcommand{\rmb}{{\rm b}}

\newcommand{\rmd}{{\rm d}}

\newcommand{\rmt}{{\rm t}}

\newcommand{\BBR}{{\mathbb R}}

\newcommand{\SC}{{\mathcal C}}

\newcommand{\SF}{{\mathcal F}}
\newcommand{\SG}{{\mathcal G}}

\newcommand{\SSS}{{\mathcal S}}

\newcommand{\SV}{{\mathcal V}}

\let\labelindent\relax
\usepackage{enumitem,amssymb}
\newlist{todolist}{itemize}{2}
\setlist[todolist]{label=$\square$}
\usepackage{pifont}

\title{Constraint Enforcing Adaptive Control of Nonlinear Systems}
\title{Constraint Enforcing Adaptive Control of Nonlinear Systems with Unknown Input Effectiveness Parameter}
\title{Adaptive State-Constrained Tracking Control of Nonlinear Systems \\ via Constraint Lifting}
\title{Adaptive Constraint-Lifting Control with Stability and Invariance Guarantees}
\author{
Jhon Manuel Portella Delgado
and
Ankit Goel
\thanks{Jhon Manuel Portella Delgado and Ankit Goel are with the Department of Mechanical Engineering, University of Maryland, Baltimore County, 1000 Hilltop Circle, Baltimore, MD 21250. {\tt\small jportella, ankgoel@umbc.edu}}%
\thanks{Jhon Manuel Portella Delgado is now with the Department of Aerospace Engineering, University of Michigan, Ann Arbor, MI 48109, USA.}
}

\begin{document}

\maketitle

%%%%%%%%%%%%%%%%%%%%%%%%%%%%%%%%%%%%%%%%%%%%%%%%%%%%%%%%%%%%%%%%%%%%%%%%%%%%%%%%
\begin{abstract}

This paper develops an adaptive tracking controller for a class of nonlinear systems with parametric uncertainty subject to state constraints.
The system is characterized by a strict-feedback structure with unknown parameters entering both the drift and input channels.
The objective is to design a control law, without knowledge of the unknown parameters, that guarantees closed-loop stability, achieves desired tracking performance, and ensures forward invariance of a prescribed safe set.
An adaptive constraint-lifting framework is developed that transforms the constrained control problem into an equivalent unconstrained representation, enabling recursive controller synthesis in lifted coordinates. 
The proposed design integrates parameter estimation with constraint enforcement without requiring online optimization.
A Lyapunov-based stability analysis, combined with the Barbashin--Krasovskii--LaSalle invariance principle, establishes boundedness of all closed-loop signals, asymptotic convergence of the system states to the desired equilibrium, and forward invariance of the safe set under uncertainty. 
In particular, the analysis characterizes the largest invariant set of the closed-loop system and guarantees convergence despite unknown parameters.
The effectiveness of the proposed approach is demonstrated on a DC motor system with uncertain parameters, illustrating accurate tracking performance and safe operation.

% This paper studies adaptive tracking control for a class of nonlinear systems with parametric uncertainty subject to state constraints.
% The system is characterized by a strict-feedback structure with unknown parameters entering both the drift and input channels.
% The objective is to design a control law, without knowledge of the uncertain parameters, that guarantees closed-loop stability, achieves desired tracking performance, and ensures forward invariance of a prescribed safe set.
% We develop an adaptive control framework that integrates parameter estimation with a systematic constraint-enforcement mechanism, enabling simultaneous handling of uncertainty and state constraints. 
% The proposed approach does not require exact knowledge of system parameters and ensures that all closed-loop signals remain bounded while the state trajectory satisfies the prescribed constraints. 
% Rigorous analysis establishes stability and constraint satisfaction guarantees for the resulting closed-loop system.
% The effectiveness of the proposed method is demonstrated on a DC motor system with uncertain parameters.
% The results illustrate accurate tracking performance and safe operation despite unkwnon parameters.
\end{abstract}
% \keywords{one, two, three, four}
\textit{\bf keywords:} 
Adaptive control,
State-constrained systems,
Nonlinear systems,
Lyapunov stability.

%%%%%%%%%%%%%%%%%%%%%%%%%%%%%%%%%%%%%%%%%%%%%%%%%%%%%%%%%%%%%%%%%%%%%%%%%%%%%%%%
% \section{List of papers}
% \begin{enumerate}
%     % \item Review of all the stability theorems for continuous and discrete systems
%     % \item Application of the constraint lifting algorithm in an actual quadcopter or the new ROS robot, or both
%     % \item Analysis of the limit cycle that appears in the application of fixed-time estimators $--->$ this may be linked with my research with Professor Bernstein
%     % \item Extend the constraint-lifting framework to include control constraints
%     % \item Extension of the constraint-lifting framework to include Asymmetric bounds
%     % \item Work on a recovery control for the constraint-lifting framework
%     \item Use the constraint lifting framework with our adaptive control
%     % \item See the possibility of writing a book about the constraint-lifting framework
% \end{enumerate}

\section{Introduction}

The enforcement of state constraints in nonlinear systems is closely related to the problem of forward invariance of a prescribed safe set.
This problem arises naturally in safety-critical control applications, where physical limitations, actuator bounds, and safety constraints must be satisfied at all times.
Set-invariance methods provide the theoretical foundation for this problem, while constrained-control architectures, such as reference governors, offer practical approaches to enforce constraints without modifying a nominal stabilizing controller \cite{Blanchini1999, Gilbert1995, Garone2017}.

A widely used strategy for constrained control is to augment a nominal controller with a supervisory mechanism that adjusts the reference or control input when constraint violation is anticipated \cite{Gilbert1995, Garone2017}.
While effective, such approaches operate as add-on schemes and do not yield an explicit stabilizing control law for the constrained system itself.
Barrier-based methods, including control barrier functions, provide an alternative by enforcing forward invariance through pointwise inequality constraints on the control input \cite{ames2016control}.
However, these methods are typically implemented via online optimization, which may introduce computational overhead and complicate stability analysis, particularly for nonlinear systems and in the presence of uncertainty.

The challenge becomes more pronounced when system dynamics are uncertain\cite{portella2026adaptive}.
Adaptive control provides a principled framework for handling parametric uncertainty, but classical adaptive designs do not directly account for state constraints.
Incorporating constraints into adaptive control laws remains nontrivial, as the interaction between parameter estimation, transient behavior, and constraint satisfaction can lead to constraint violations or loss of stability guarantees.

An alternative approach to constraint enforcement is based on nonlinear coordinate transformations that map a constrained state space into an unconstrained one, enabling the use of standard nonlinear control tools in the transformed coordinates \cite{Tee2009, Guo2014}.
While such transformation-based methods avoid online optimization, they may suffer from numerical ill-conditioning near the constraint boundary, where the transformed coordinates or their derivatives become large.
These challenges are further exacerbated in adaptive settings, where parameter-estimation dynamics can drive the system toward regions in which the transformation becomes poorly conditioned.

Motivated by these limitations, this paper develops an adaptive constraint-lifting framework for continuous-time nonlinear systems.
The central idea is to map constrained states into an unconstrained coordinate system via smooth constraint-lifting transformations constructed from strictly increasing sigmoid functions.
Controller synthesis is carried out in the lifted coordinates using a recursive adaptive design, while the inverse transformation guarantees satisfaction of the original state constraints.

% The contributions of this paper are as follows.
% First, an adaptive constraint-lifting control framework is developed for a class of continuous-time strict-feedback systems with parametric uncertainty, transforming the constrained adaptive control problem into an equivalent unconstrained one.
% Second, conditions on the controller gains and adaptation laws are derived to ensure both asymptotic stability of the closed-loop system and forward invariance of the admissible domain of the lifting transformations.
% Third, the analysis establishes the boundedness of the parameter estimates and guarantees satisfaction of the constraints during both transient and steady-state operation.

The contributions of this paper are as follows.
First, an adaptive constraint-lifting control framework is developed for a class of continuous-time strict-feedback systems with parametric uncertainty, transforming the constrained adaptive control problem into an equivalent unconstrained one.
Second, a recursive adaptive controller is constructed in the lifted coordinates, together with adaptation laws that compensate for uncertainty in both the drift and input channels.
Third, a Lyapunov-based stability analysis, combined with the Barbashin--Krasovskii--LaSalle invariance principle, establishes the asymptotic convergence of the system states, the boundedness of all closed-loop signals, and the forward invariance of the prescribed safe set.
In particular, the analysis characterizes the largest invariant set of the closed-loop system and guarantees convergence despite unknown parameters.
The proposed approach yields an explicit adaptive control law, avoids online optimization, and ensures constraint satisfaction under uncertainty.
Numerical simulations illustrate the effectiveness of the proposed framework on representative nonlinear systems.

The remainder of the paper is organized as follows. 
Section \ref{sec:problem} formulates the problem and introduces the class of nonlinear systems considered in this work. 
Section \ref{sec:Constraint_Lifting} presents the constraint-lifting framework, develops the adaptive control law, and provides the Lyapunov-based stability analysis. 
Section \ref{sec:numerical} demonstrates the effectiveness of the proposed approach through numerical simulations. 
Finally, Section \ref{sec:conclusions} concludes the paper.

% \newpage
\section{Problem Formulation}
\label{sec:problem}
Consider the system
\begin{align}
    \dot{x}_1 
        &= 
            g_1(x_1)  x_2, 
    \label{eq:x1_dot}
    \\
    \dot{x}_2 
        &= 
            f_2(x_1, x_2)\,\Theta_1 
            + 
            g_2(x_1, x_2)\, u\,\Theta_2,
    \label{eq:x2_dot}
\end{align}
where $x_1 \in \mathbb{R},$ and $x_2 \in \mathbb{R}$ are the states, $u \in \mathbb{R}$ is the control input, and $\Theta_1, \Theta_2 \in \mathbb{R}$ are unknown parameters, and the functions $g_1: \mathbb{R} \to \mathbb{R},$ $f_2: \mathbb{R} \times \mathbb{R} \to \mathbb{R},$ and $g_2: \mathbb{R} \times \mathbb{R} \to \mathbb{R}.$
The control objective is to design a control law $u$ without the knowledge of the unknown parameters $\Theta_1$ and $\Theta_2$, such that,
$i)$ the closed-loop system is stable,
$ii)$ a desired tracking performance is achieved, and
$iii)$ the state trajectory satisfies $x \in \SSS$, where $\SSS$ is a user-specified safe set.
The safe set $\SSS$ is assumed to be a compact hyper-rectangle of the form $\SSS = (-\overline{x}_1, \overline{x}_1) \times (-\overline{x}_2, \overline{x}_2)$.

\begin{assumption}
    \label{ass:function_zero}
    The function $f_2(x_1,x_2) = 0$ if and only if $x_2 = 0.$
\end{assumption}

\begin{assumption}
    \label{ass:function_invert}
    The matrices $g_1(x_1)$ and $g_2(x_1,x_2)$ are nonsingular for $(x_1,x_2) \in \SSS.$ 
\end{assumption}

\begin{remark}
Assumption \ref{ass:function_zero} ensures that the uncertain dynamics associated with $\Theta_1$ are excited whenever $x_2 \neq 0$, and vanish only at $x_2 = 0$. 
This condition prevents loss of excitation in the system and guarantees that the effect of the uncertainty is observable whenever the state is evolving.
\end{remark}

\begin{remark}
Assumption \ref{ass:function_invert} guarantees that the system remains free of singularities within the safe set $\SSS$. 
In particular, $g_1(x_1) \neq 0$ ensures that the state $x_2$ influences the evolution of $x_1$, while $g_2(x_1,x_2) \neq 0$ ensures that the control input $u$ influences the evolution of $x_2.$
Together, these conditions ensure well-posedness of the control design and avoid loss of controllability within $\SSS$.
\end{remark}

Dynamics \eqref{eq:x1_dot}--\eqref{eq:x2_dot} arises naturally in a variety of engineering applications. 
A representative example is a DC motor driving a mechanical load, whose equations of motion are given by
\begin{align}
    \dot{\theta} 
        &= 
            \omega,
    \label{eq:x1_dc_motor}
    \\
    \dot{\omega} 
        &= 
            -
            \frac{bR - K_\rmb K_\rmt}{JR}
            \,
            \omega 
            + 
            \frac{K_\rmt}{JR}\, 
            V,
    \label{eq:x2_dc_motor}
\end{align}
where $\theta$ and $\omega$ are the shaft angle and angular velocity, $V$ is the applied voltage, $J$ is the rotor inertia, $b$ is the viscous damping coefficient, $R$ is the armature resistance, and $K_\rmt$, $K_\rmb$ are the torque and back-EMF constants, respectively. 
This model fits \eqref{eq:x1_dot}--\eqref{eq:x2_dot} with 
$x_1 = \theta$, $x_2 = \omega$, 
$g_1(x_1) = 1$, 
$f_2(x_1, x_2) = \omega$, 
$g_2(x_1, x_2) = 1$, 
$u = V$, 
and unknown parameters
\begin{align}
    \Theta_1 
        = 
            -\frac{bR - K_\rmb K_\rmt}{JR}, 
    \qquad 
    \Theta_2 
        = 
            \frac{K_\rmt}{JR},
\end{align}
which are uncertain in practice due to manufacturing tolerances, temperature variation, and component wear.

Note that the DC motor system fullfills Assumptions \ref{ass:function_zero} and \ref{ass:function_invert} are fulfilled since $f_2 = \omega = 0$ if and only if $x_2 = \omega = 0,$ and $g_1 = g_2 = 1,$ are invertible.

\section{Constraint Lifting Framework}
\label{sec:Constraint_Lifting}

This section presents the constraint-lifting framework for transforming the state-constrained control problem into an equivalent unconstrained representation. 
The original states are first normalized and mapped into a lifted coordinate space via smooth sigmoid-based transformations, enabling controller synthesis without explicit constraint handling. 
The resulting dynamics in the lifted coordinates are derived, and key structural properties required for control design are established.
Based on this formulation, a recursive adaptive control law is developed, followed by a Lyapunov-based analysis that guarantees closed-loop stability, boundedness of all signals, and forward invariance of the prescribed safe set.

The states to be constrained are \textit{lifted} using a constraint-lifting transformation to obtain an unconstrained representation in a lifted state space.
The sigmoid functions and their inverses to lift the constraints are described in detail in \cite{portella2025constrained}.

First, the states $(x_1,x_2)$ are mapped to intermediate variables $(\chi_1,\chi_2)$ such that the bounds on the components of $x_1$ and $x_2$ are normalized to $\pm1$. 
Consequently, whenever $(x_1,x_2)\in\SSS$, each component of $(\chi_1,\chi_2)$ lies in $(-1,1)$, that is, $(\chi_1,\chi_2)\in\SC$.
The normalized variables $(\chi_1,\chi_2)$ are then mapped to $(z_1,z_2)\in\BBR^n$ using a constraint-lifting function. 
Controller synthesis is performed in the lifted $z$-coordinates for the transformed state $x_{1\rmd}$.
Finally, a transformation from $(z_1,z_2)$ to $(\zeta_1,\zeta_2)$ is introduced to simplify both the inverse mapping from $z$ to $x$ and the subsequent stability analysis.

The sequence of transformations from $x$ to $z$, together with the intermediate variables $\chi$ and $\zeta$, is shown in Figure \ref{fig:state_transformation}.

\begin{figure}[!ht]
\centering
    {%
    \begin{tikzpicture}[>={stealth'}, line width = 0.25mm]

    % Reference input
    \node [input, name=ref]{};
    
    % Adaptive Controller block
    \node [smallblock, fill=green!20, rounded corners, right = 0.5cm of ref , 
           minimum height = 0.6cm, minimum width = 0.7cm] 
           (controller) {$\dfrac{x}{\overline{x}}$};
    
    \node [smallblock, rounded corners, right = 2cm of controller, 
           minimum height = 0.6cm , minimum width = 0.7cm] 
           (integrator)            {$\overline x\,\phi(\chi)$};
    
    \node [smallblock, fill=green!20, rounded corners, below = 0.5cm of integrator, 
           minimum height = 0.6cm , minimum width = 0.7cm] 
           (system) {$\dfrac{z}{\overline{x}}$};
           
    \node [smallblock, rounded corners, below = 0.5cm of controller, 
           minimum height = 0.6cm , minimum width = 0.7cm] 
           (integrator2)            {$\overline x\,\psi(\zeta)$};

    % Output node
    \node [output, right = 1.0cm of system] (output) {};
    
    % Midpoint for feedback to controller
    
    \node [input, left = 1.0cm of controller] (reference) {};
    
    % Arrows
    \draw [->] (controller) -- node [above] {$\chi\in \SC$} (integrator);
    \draw [->] (integrator.0) -- +(1,0) node [above, xshift = -1em] {$z\in\BBR^n$} |-  (system.0);
    \draw [->] (system.180)  -- node [above] {$\zeta \in \BBR^n$} (integrator2.0);
    \draw [->] (integrator2.180) -- +(-1,0) |- node [above, xshift = 1em] {$x \in \SSS$} (controller.180);
    \end{tikzpicture}
    }  
    \caption{
        State transformations with the constraint-lifting function and its inverse used in the control synthesis.
    }
    \label{fig:state_transformation}
\end{figure}
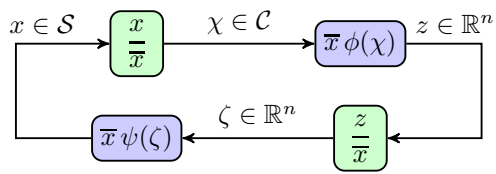
Define the normalized states
\begin{align}
    \chi_1
        &\isdef
            \overline{x}_1^{-1}
            x_1,
    \\
    \chi_2
        &\isdef
            \overline{x}_2^{-1}
            x_2.
\end{align}
Note that the state variables $\chi_1$ and $\chi_2$ are defined to map the constraint boundaries to $\pm 1,$ that is, when the constraints are satisfied, then, $\chi_{1}, \chi_{2} \in \SC.$

Next, define
\begin{align}
    z_1 
        &\isdef
            \overline{x}_1
            \phi(\chi_1),
    \label{eq:x1_to_z1}
    \\
    z_2 
        &\isdef
            \overline{x}_2
            \phi(\chi_2),
    \label{eq:x2_to_z2}
\end{align}
where $\phi \colon \SC  \to \BBR^n $ is a constraint-lifting function as defined in \cite{portella2025constrained}.
\begin{remark}
    For simplicity, the same constraint-lifting function $\phi$ is used for both $z_1$ and $z_2$. 
    This choice is made without loss of generality, as the formulation readily extends to the case where distinct constraint-lifting functions are assigned to each state.
\end{remark}

Next, define
\begin{align}
    \zeta_{1}
        &\isdef
            \overline{x}_1^{-1}
            z_1,
    \label{eq:zeta1_def}
    \\
    \zeta_{2}
        &\isdef
            \overline{x}_2^{-1}
            z_2.
    \label{eq:zeta2_def}
\end{align}
The state variables $\zeta_1$ and $\zeta_2$ are defined to simplify the controller synthesis, as described below. 

Finally, define
\begin{align}
    x_1 
        &=
            \overline{x}_1
            \psi(\zeta_1),
    \label{eq:z1_to_x1}
    \\
    x_2 
        &=
            \overline{x}_2
            \psi(\zeta_2),
    \label{eq:z2_to_x2}
\end{align}
where $\psi:\mathbb{R}^n \to \SSS$ is an inverse of the constraint-lifting function as defined in \cite{portella2025constrained}.

With the state transformations defined above, the dynamics \eqref{eq:x1_dot}, \eqref{eq:x2_dot} in terms of the transformed state $z$ is
\begin{align}
    \dot{z}_1
        &=
            \SG_1(z_1,z_2), 
    \label{eq:z1_dot}
    \\
    \dot{z}_2
        &=
            \SF_2(z_1,z_2)\Theta_1  + \SG_2(z_1,z_2) u\,\Theta_2,
    \label{eq:z2_dot}
\end{align}
where 
\begin{align}
    \SG_1(z_1,z_2)
        &\isdef 
            \Phi(\zeta_1) 
            \psi(\zeta_2), 
    \label{eq:SG1def}
    \\
    \SF_2(z_1,z_2)
        &\isdef
            \left[ 
                \partial_{\chi_2} 
                \phi(\psi(\zeta_2))
            \right]
                f_2
                (
                    \overline{x}_1
                    \psi(\zeta_1),
                    \overline{x}_2
                    \psi(\zeta_2)
                ), 
    \label{eq:SF2def}
    \\
    \SG_2(z_1,z_2)
        &\isdef
            \left[ 
                \partial_{\chi_2} 
                \phi(\psi(\zeta_2))
            \right]
            g_2
                (
                    \overline{x}_1
                    \psi(\zeta_1),
                    \overline{x}_2
                    \psi(\zeta_2)
                ),
    \label{eq:SG2def}
\end{align}
and
\begin{align}
    \Phi(\zeta_1) 
        \isdef
            \left[ \partial_{\chi_1} \phi(\psi(\zeta_1)) \right] 
            g_1(\overline{x}_1\psi(\zeta_1))
            \overline{x}_2.
    \label{eq:Phi_def}
\end{align}
Note that $\zeta_1$ and $\zeta_2$ are given by \eqref{eq:zeta1_def}, \eqref{eq:zeta2_def}.
Algebraic details of the derivation of dynamics in $z$ coordinates are shown in Appendix \ref{ap:derivation of z_dynamics}.

The following propositions establish key properties of the functions $\SG_1$, $\SF_2$, and $\SG_2$, which are subsequently used in the controller design and stability analysis.

\begin{proposition}
    \label{prop:equil_point_z1_dot}
    Consider $\SG_1(z_1,z_2) \in \BBR^{n}$ given by \eqref{eq:SG1def}.
    Then, $\SG_1(z_1,0) = 0.$
\end{proposition}
\begin{proof}
    Since $\psi(0)=0$ by definition, it follows from \eqref{eq:SG1def} that 
    $\SG_1(z_1,0) = \Phi(\overline{x}_1^{-1} z_1) \psi(0)=0.$
\end{proof}

\begin{proposition}
    \label{prop:equil_point_z2_dot}
    Consider $\SF_2(z_1,z_2) \in \BBR^n$ given by \eqref{eq:SF2def}. 
    Then, $\SF_2(z_1,0) = 0.$
\end{proposition}
\begin{proof}
    Since $\psi(0)=0$ by definition, it follows from Assumption \ref{ass:function_zero} that $f_2
                (
                    \overline{x}_1
                    \psi(\zeta_1),
                    0
                ) = 0,$
    which in turn implies $\SF_2(z_1,0) = 0.$
\end{proof}

\begin{proposition}
    \label{prop_SG2_invertible}
    Consider $\SG_2(z_1,z_2) \in \BBR^{n\times n}$ given by \eqref{eq:SG2def}.
    Then, $\SG_2(z_1,z_2)$ is nonsingular. 
\end{proposition}
\begin{proof}
    It follows from Assumption \ref{ass:function_invert} that $g_2
                (
                    \overline{x}_2
                    \psi(\zeta_1),
                    \overline{x}_2
                    \psi(\zeta_2)
                ) $ is nonsingular.
    Since the product of two nonsingular matrices is a nonsingular matrix, it follows that $\SG_2(z_1,z_2)$ is nonsingular.
\end{proof}

\begin{proposition}
    \label{prop:PHI_invertible}
    Consider $\Phi(\zeta_1) \in \mathbb{R}^{n \times n}$ given by \eqref{eq:Phi_def}. 
    Then, $\Phi(\zeta_1)$ is nonsingular.
\end{proposition}
\begin{proof}
    It follows from Assumption \ref{ass:function_invert} that
    $g_1
    (
        \overline{x}_2
        \psi(\zeta_1),
    ) $
    is nonsingular.
    Furthermore, $\overline{x}_2$ is positive.
    Then, it follows that $\Phi(\zeta_1)$ is nonsingular.
\end{proof}

\subsection{Controller Synthesis}
\label{sec:ControlSynthesis}
Let $x_{1\rmd} \in \SSS$ be the desired value of $x_1.$
Let $\chi_{1\rmd} \in \SC$ and $z_{1\rmd}$ be the corresponding desired state values, that is, 
\begin{align}
    \chi_{1\rmd} 
        &\isdef 
        \overline{x}_1^{-1}
        x_{1\rmd},
        \\
    z_{1\rmd} 
        &\isdef 
            \overline{x}_1
            \phi(\chi_{1\rmd}).
    \label{eq:z1d}
\end{align}
Define the tracking error 
\begin{align}
    e_1 
        &\isdef 
            z_1 - z_{1\rmd}.
    \label{eq:e1_def}
\end{align}
\subsubsection{\texorpdfstring{$e_1$}{e1} stabilization}
Consider the function
\begin{align}
    V_1
        &\isdef
            \dfrac{1}{2}
            e_1^2.
    \label{eq:V1def}
\end{align}
Differentiating \eqref{eq:V1def} and using \eqref{eq:z1_dot} yields
\begin{align}
    \dot{V}_1
        &=
            e_1
            \SG_1(z_1,z_2).
    \label{eq:V1dot}
\end{align}

If $z_2$ is chosen such that $\SG_1(z_1,z_2) = - k e_1,$ where $k_1 > 0,$ then $\dot V_1 < 0.$
However, since $z_2$ is not the control input, $\SG_1(z_1,z_2)$ cannot be arbitrarily chosen. 
Instead, we define
\begin{align}
    e_2
        &\isdef
            \SG_1(z_1,z_2)
            -
            (- k_1e_1)
            =
                \Phi(\zeta_1) 
                \psi(\zeta_2)
                +
                k_1e_1.
    \label{eq:e2_def}
\end{align}
and design the control law to ensure that $e_2 $ converges to $ 0,$ thus implying that $\SG_1(z_1,z_2)$ converges to $ - k_1e_1.$ 
Substituting $\SG_1(z_1,z_2)$ from \eqref{eq:e2_def} into \eqref{eq:V1dot} yields
\begin{align}
    \dot{V}_1
        &=
            - k_1 e_1^2
            + e_2e_1.
\end{align}

\subsubsection{\texorpdfstring{$e_2$}{e2} stabilization}
Define $p_2 \isdef \Theta_2^{-1},$ and $\hat{p}_2$ its corresponding estimate.
Next, consider the function
\begin{align}
    V
        &\isdef
            V_1
            +
            % \sum_{i=1}^n
            \SV(\zeta_{2})
            +
            \dfrac{1}{2}
            \gamma^{-1}
            |\Theta_2|
            \left(
                \hat{p}_2
                -
                p_2
            \right)^2
            +
            \dfrac{1}{2}
            \alpha^{-1}
            \left(
                \Theta_1
                -
                \hat{\Theta}_1
            \right)^2,
    \label{eq:LyapCandidate}
\end{align}
where $\gamma >0,$ $\alpha > 0,$ and $\SV$ is a sigmoid integral function as introduced in \cite{portella2025constrained}.
Note that
\begin{align}
    \rmd_t
    % \sum_{i=1}^n
            \SV(\zeta_{2})
        &= 
            \psi(\zeta_2)
            % ^{\rmT}
            \overline{x}_2^{-1}
            \dot{z}_2. 
\end{align}
Thus,
\begin{align}
    \dot{V}
        &=
            - k_1e_1^2
            + e_2e_1
            + 
            \psi(\zeta_2)
            \overline{x}_2^{-1}
        % \nn \\
        % & \quad
            % \Big(
            \SF_2(z_1,z_2)\Theta_1
        \nn \\
        & \quad
            +
            \psi(\zeta_2)
            \overline{x}_2^{-1}
            \SG_2(z_1,z_2)u\Theta_2
            % \Big).
            +
            \gamma^{-1}
            |\Theta_2|
            \left(
                \hat{p}_2
                -
                p_2
            \right)
            \dot{\hat{p}}_2
        \nn \\
        & \quad
            -
            \alpha^{-1}
            \left(
                \Theta_1
                -
                \hat{\Theta}_1
            \right)
            \dot{\hat{\Theta}}_1.
\end{align} 
Choose
\begin{align}
    \dot{p}_2
        &=
            \gamma
            \,
            \sigma(\Theta_2)
            \psi(\zeta_2)
            \left(
                \SF_2(z_1,z_2)
                \hat{\Theta}_1
                +
                \Phi(\zeta_1)
                k_2
                e_2
            \right),
    \label{eq:p2_dot}
    \\
    \dot{\hat{\Theta}}_1
        &=
            \alpha
            \,
            \psi(\zeta_2)
            \SF_2(z_1,z_2),
    \label{eq:theta_1_hat_dot}
    \\
    u
        &=
            - 
            \overline{x}_2
            \SG_2(z_1,z_2)^{-1}
            \hat{p}_2
            \left(
                \SF_2(z_1,z_2)
                \hat{\Theta}_1
                +
                \Phi(\zeta_1)
                k_2
                e_2
            \right),
    \label{eq:u}
\end{align}
where $\sigma(\Theta_2) \isdef {\rm{sign}}(\Theta)_2.$
Then,
\begin{align}
    \dot{V}
        &=
           - 
           k_1e_1^2
           + 
           e_2e_1
           -
           \psi(\zeta_2)
           \Phi(\zeta_1)
           k_2
           e_2.
\end{align}
Substituting $\psi(\zeta_2)$ from \eqref{eq:e2_def} in the equation above yields
\begin{align}
    \dot{V}
        &=
            - k_1e_1^2
            + e_2e_1
            - 
            (e_2 - k_1e_1)
            k_2e_2
\end{align}
% Finally, choosing 
% \begin{align}
%     u
%         &=
%             - \SG_2(z_1,z_2)^{-1}
%             \left(
%                 \SF_2(z_1,z_2)
%                 +
%                 k_2
%                 D(\overline{x}_2)
%                 \Phi^{\rmT}(\zeta_1)
%                 e_2
%             \right)
%     \label{eq:u}
% \end{align}
yields
\begin{align}
    \dot{V}
        &=
            - k_1e_1^2
            + (1 + k_1k_2)e_2e_1
            - k_2e_2^2 
        \nn \\
        &=
            - k_1e_1^2
            + (1 + k_1k_2)e_2e_1
            - k_2e_2^2
        \nn \\
        & \quad
            + 2\sqrt{k_1k_2}e_2e_1
            - 2\sqrt{k_1k_2}e_2e_1
        \nn \\
        &=
            - 
            \left(
                \sqrt{k_1}e_1 - \sqrt{k_2}e_2
            \right)^2
        % \nn \\
        % & \quad
            + (1 + k_1k_2 - 2\sqrt{k_1k_2})e_2e_1.
            \label{eq:V_dot}
\end{align}

Note that if $k_2 = k_1^{-1},$ then,
\begin{align}
    \dot{V}
            &=
            - 
            \left(
                \sqrt{k_1}e_1 - \sqrt{k_2}e_2
            \right)^2
            % \left(\sqrt{k_1}e_1 - \sqrt{k_2}e_2\right).
            \label{eq:V_dot_k2condition}
\end{align}
Furthermore, if $k_1 e_1 = e_2,$ then $\dot V = 0.$

\subsection{Stability Analysis}
\label{sec:stability_analysis}
The stability analysis of the closed-loop dynamics with the proposed control law uses the Barbashin-Krasovskii-LaSalle invariance principle, which originally appeared as Theorems 1 and 2 in \cite{lasalle1960some}.
In particular, the Barbashin-Krasovskii-LaSalles's invariance principle is used to prove the asymptotic stability of the equilibrium point.

Propositions \ref{prop:equilibriumPoint}, \ref{prop:Omega_def}, \ref{prop:E_def}, and \ref{prop:M_def} introduced below define the sets used in the stability analysis and establish their properties.

\begin{proposition}
    \label{prop:equilibriumPoint}
    [\textbf{Equilibrium solution.}]
    Consider the system \eqref{eq:z1_dot}, \eqref{eq:z2_dot}.
    Let $z_{1\rmd} \in \BBR^n.$
    Then, the solution
    \begin{align}
        (z_1,z_2,u) 
            =
            \left(
                z_{1\rmd},
                0,
                0
            \right)
        \label{eq:equilibrium_point}
    \end{align}
    is an equilibrium point.
\end{proposition}
\begin{proof}
    If $z_2=0,$ then it follows from Proposition \ref{prop:equil_point_z1_dot} that $\SG_1(z_1,z_2)=0,$ which implies that $\dot z_1 = 0.$
    Furthermore, it follows from Proposition \ref{prop:equil_point_z2_dot} that $\SF_2(z_1,z_2)=0.$
    Moreover, if $u = 0,$ then $\dot z_2 = 0,$ implying that $(z_1, z_2, u) \equiv (z_{1\rmd},0,0)$ is an equilibrium point.
\end{proof}
\begin{proposition}
     \label{prop:Omega_def}
    Consider the system \eqref{eq:z1_dot}, \eqref{eq:z2_dot}.
    Consider the Lyapunov candidate function \eqref{eq:LyapCandidate}.
    Let $k_1>0,$ $k_2 = k_1\inv,$ $\gamma>0,$ and $\alpha>0.$
    Let $\ell > 0.$ 
    Define the sublevel set
    \begin{align}
        \Omega
            \isdef
                \{
                    z_1,z_2,u,\hat{p}_2,\hat{\Theta}_1 \in
                    \mathbb{R}
                    : V
                    \leq \ell
                \}.
        \label{eq:Omega_def}
    \end{align}
    Then, $\Omega$ is bounded and a positively invariant set.
\end{proposition}
\begin{proof}
    Since $V$ is radially unbounded and $\Omega$ is defined by $V \leq \ell,$ it follows from Theorem $2$ in \cite{lasalle1960some} that $\Omega$ is bounded. 
    Since $\dot{V} \leq 0$ and $V > 0,$ $V$ is non-increasing along the trajectory of the system. 
    Therefore, any trajectory starting in $\Omega$ cannot leave $\Omega,$ thus implying that $\Omega$ is a positively invariant set.
\end{proof}

\begin{proposition}
    \label{prop:E_def}
    Consider the system \eqref{eq:z1_dot}, \eqref{eq:z2_dot}.
    Consider the function \eqref{eq:LyapCandidate} and the positively invariant set $\Omega$, given in \eqref{eq:Omega_def}.
    Let $k_2 = k_1\inv.$
    Define $E \subset \Omega$ such that
    $z_2 = 0.$
    Then, $E$ is the set of all the points in $\Omega$ such that $\dot{V} = 0.$
\end{proposition}
\begin{proof}
    Note that it follows from \eqref{eq:zeta2_def} that $\zeta_2 = 0$ if and only if $z_2 = 0.$
    Next, $\zeta_2 = 0$ implies that $\psi(\zeta_2) = 0,$ which further implies that $\Phi(\zeta_1)  \psi(\zeta_2) = 0.$
    Next, it follows from \eqref{eq:e2_def} that $k_1e_1 = e_2.$
    Finally, since $k_2 = k_1\inv,$ it follows from \eqref{eq:V_dot} that $\dot{V} = 0$ if and only if $k_1e_1 = e_2.$
\end{proof}

\begin{proposition}
    \label{prop:M_def}
    Consider the system \eqref{eq:z1_dot}, \eqref{eq:z2_dot}.
    Consider the function \eqref{eq:LyapCandidate}.
    Let $k_2 = k_1\inv.$
    Let $M \subset E$ such that
    $u = 0.$ 
    Then, $M$ is the largest positively invariant set in $E.$
\end{proposition}

\begin{proof}
    It follows from Proposition \ref{prop:equilibriumPoint} that the point $ z_1 = z_{1\rmd}, \, z_2 = u = 0 $ is an equilibrium of the system. 
    Hence, the set of all equilibrium points constitutes a positively invariant subset of $ E $. 
    Next, consider the case where $ u \neq 0 $ and $ z_2 = 0 $. 
    From \eqref{eq:z2_dot} and the fact that $ \SG_2(z_1, z_2) $ is nonsingular for all 
    $ z_1, z_2 \in \mathbb{R}^n $, it follows that $ \dot{z}_2 \neq 0 $. 
    Consequently, the trajectory leaves the set $ E $. 
    Therefore, $ M \subset E $ is the largest positively invariant subset of $ E $.
\end{proof}

\begin{theorem}
    Consider the system \eqref{eq:x1_dot}, \eqref{eq:x2_dot} and the system \eqref{eq:z1_dot}, \eqref{eq:z2_dot}.
    Assume that the initial conditions of \eqref{eq:x1_dot}, \eqref{eq:x2_dot} are inside the safe set, that is, 
    $(x_1(0),x_2(0)) \in \SSS,$ and let $|x_{1\rmd}|<\overline{x}_1.$
    Consider the control law \eqref{eq:u}, with $k_2 = k_1^{-1}.$
    Then,
    \begin{align}
        \lim_{t\to\infty}z_1(t) 
            &= 
                z_{1\rmd},
        \label{eq:z1limit}
        \\
        \lim_{t\to\infty}z_2(t) 
            &= 
                0,
        \\
        \lim_{t\to\infty}u(t) 
            &= 
                0,
        \label{eq:ulimit}
        \\
        \lim_{t\to\infty}x_1(t) 
            &= 
                x_{1\rmd},
        \label{eq:x1limit}
        \\
        \forall t\geq0,
        (x_1(t), x_2(t)) 
            &\in
                \SSS,
        \label{eq:safeset_FI}
    \end{align}
    and the estimates $\hat p_2(t), \hat \Theta_1(t)$ remain bounded.
\end{theorem}
\begin{proof}
    Consider the set $\Omega$ defined in Proposition \ref{prop:Omega_def} and the set $E$ defined in Proposition \ref{prop:E_def}. 
    It follows from Proposition \ref{prop:M_def} that $M$ is the largest positively invariant set in $E.$
    It follows from Theorem 2 in \cite{lasalle1960some} that any solution starting in $\Omega$ converge to $M$ as $t\to\infty,$ which implies \eqref{eq:z1limit}--\eqref{eq:ulimit}, and the estimates $\hat p_1(t), \hat \theta_1(t), \hat \vartheta_1(t), \hat \theta_2(t), \hat p_2(t) $ remain bounded.
    Next, it trivially follows from \eqref{eq:z1d} and \eqref{eq:z1limit} that $\lim_{t\to\infty}x_1(t) = x_{1\rmd}.$
    Finally, since the map $\psi$ is a bijection, it follows that $\forall t\geq0,
        (x_1(t), x_2(t)) \in \SSS.$
\end{proof}

\begin{remark} 
It can not be concluded that $\hat{\Theta}_1$ converges to $\theta_1$ and  $\hat{p}_2$ converges to $p_2$.
However, they remain bounded.
\end{remark}

% \newpage
\section{Numerical Simulation}
\label{sec:numerical}
In this section, we illustrate the proposed framework on the DC motor dynamics described by \eqref{eq:x1_dc_motor}--\eqref{eq:x2_dc_motor}. 
% The control objective is to regulate the motor angle to a desired reference in the presence of unknown physical parameters, while ensuring that the system state remains within a prescribed safe set. 
The control objective is to regulate the shaft position (angle) to a desired reference in the presence of unknown physical parameters, while ensuring that the system state remains within a prescribed safe set.
The safe set is defined as
\begin{align}
    \SSS \isdef (-2,2) \times (-1,1).
\end{align}

For the constraint-lifting construction, we use the sigmoid function $\psi(x) = \tanh(x)$ and its associated integral
\begin{align}
    \SV(\zeta_2) = \log\big(\cosh(\zeta_2)\big),
\end{align}
as described in \cite{portella2025constrained}.

% In this work, we use the sigmoid function $\psi(x) = \tanh(x),$ and its corresponding sigmoid integral function $\SV = \log(\cosh(\zeta_2))$, as presented in \cite{portella2025constrained}.

% The adaptation laws, with this choice of sigmoid function are
The adaptation laws corresponding to the chosen sigmoid function are given by
\begin{align}
    \hat{p}_2
        &=
            -
            \gamma
            \sigma(\Theta_2)
            \tanh{\zeta_2}
            \Big(
                \tanh{(\zeta_2)}
                \cosh^2{(\zeta_2)}
                \hat{\Theta}_1
    \nn \\
    & \quad
                +
                \overline{x}_2
                \cosh^2{(\zeta_1)}
                k_1^{-1}
                e_2
            \Big),
    \\
    \hat{\Theta}_1
        &=
            \alpha
            \tanh^2{(\zeta_2)}
            \cosh^2{(\zeta_2)},
\end{align}
and the corresponding controller is given by
\begin{align}
    u
        &=
            -
            \overline{x}_2
            \sech^2{(\zeta_2)}
            \hat{p}_2
            \Big(
                \cosh^2{(\zeta_2)}
                \tanh{(\zeta_2)}
                \hat{\Theta}_1
    \nn \\
    & \quad
                +
                \overline{x}_2
                \cosh^2{(\zeta_1)}
                k_1^{-1}
                e_2
            \Big).
\end{align}
Note that the initial condition $\hat{p}_2(0)$ cannot be chosen as zero, since this would result in $u \equiv 0$ at the initial time.

The initial conditions are selected as
\begin{align}
    x_1(0) = 0, \quad x_2(0) = 0.9,
\end{align}
with parameter estimates initialized as
\begin{align}
    \hat{p}_2(0) = 1, \quad \hat{\Theta}_1(0) = 0.
\end{align}
The control gain is set to $k_1 = 1$, and the adaptation gains are chosen as $\gamma = \alpha = 1$.
% The constraint bounds are given by $\overline{x}_1 = 2$ and $\overline{x}_2 = 1$.
% 
To evaluate performance near the boundary of the safe set, the desired state is selected close to the constraint limit. 
In particular, we set $x_{1\rmd} = -1.9.$

% Note that the initial condition $\hat{p}_2(0)$ can not be set to $0.$ Otherwise, $u = 0$ inmediately.
% %
% Therefore, in this example, we set the initial states $x_1(0) = 0,$ $x_2(0) = 0.9,$ and initial estimations $\hat{p}_2(0) = 1,$ and $\hat{\Theta}_1(0) = 0.$
% %
% We set $k_1 = 1,$ and the adaptatio gains, $\gamma = \alpha = 1.$
% %
% In this example $\overline{x}_1 = 2,$ and $\overline{x}_2 = 1.$
% %
% We also, on purpose, set the desired state close to the bound $\overline{x}_1,$ that is $x_{1\rmd} = -1.9.$

Figure \ref{fig:closed_loop_motor} shows the closed-loop response of \eqref{eq:x1_dc_motor}--\eqref{eq:x2_dc_motor} under the control law \eqref{eq:u} and the adaptation laws \eqref{eq:p2_dot} and \eqref{eq:theta_1_hat_dot}. 
The results demonstrate successful tracking of $x_{1\rmd}$ while maintaining the state within the prescribed safe set, indicated by the shaded yellow region in the first plot. 
Furthermore, the state $x_2$ remains within its corresponding safe bounds, as shown by the shaded pink region in the second plot.

% Figure \ref{fig:closed_loop_motor} shows the closed-loop response of \eqref{eq:x1_dc_motor}--\eqref{eq:x2_dc_motor} with the cotroller \eqref{eq:u}, and the adaptation laws \eqref{eq:p2_dot} and \eqref{eq:theta_1_hat_dot}.
% %
% Note that the closed-loop succesfully tracks $x_{1\rmd}$ while staying inside the safe area shown in shaded yellow in the first plot.
% %
% Note also that $x_2$ is always the safe area, shown in shaded pink in the second plot.

% Figure \ref{fig:estimations_motor} shows the estimation errors $|\hat{\Theta}_1 - \Theta_1|,$ and $|\hat{p}_2 - p_2|$ in logarithmic scale, obtained from \eqref{eq:theta_1_hat_dot} and \eqref{eq:p2_dot}, respectively.
% %
% Note that, as shown in the stability analysis, neither $\hat{\Theta}_1$ nor $\hat{p}_2$ converges to the true values $\Theta_1$ and $p_2$; however, both estimates remain bounded.
Figure \ref{fig:estimations_motor} presents the estimation errors $|\hat{\Theta}_1 - \Theta_1|$ and $|\hat{p}_2 - p_2|$ on a logarithmic scale, corresponding to the update laws \eqref{eq:theta_1_hat_dot} and \eqref{eq:p2_dot}, respectively. 
Consistent with the stability analysis, neither $\hat{\Theta}_1$ nor $\hat{p}_2$ converges to the true parameters $\Theta_1$ and $p_2$; however, both estimates remain bounded.
\begin{figure}[!ht]
    \centering
    \includegraphics[width=\columnwidth]{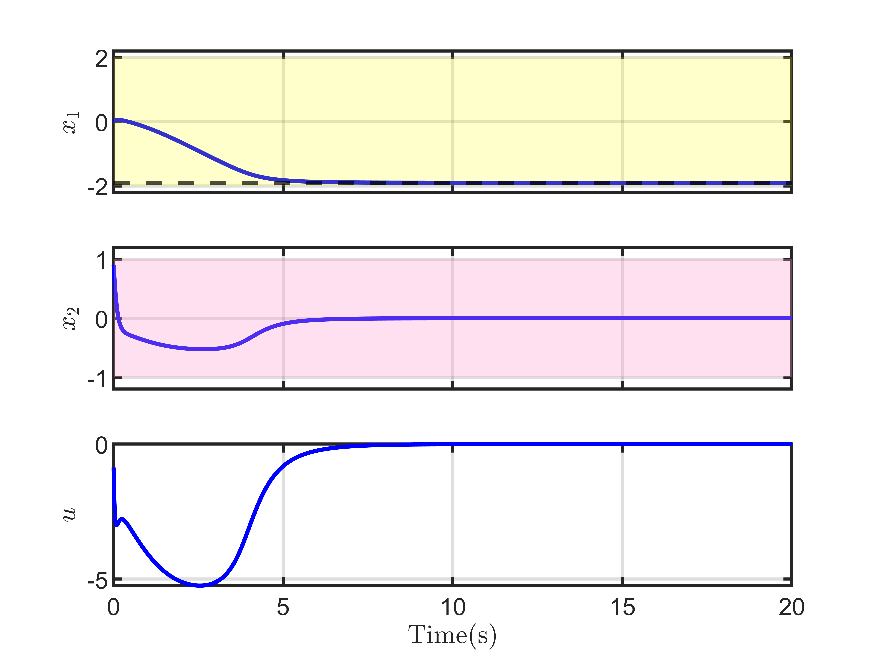}
    \caption{
    Closed-loop response of the system \eqref{eq:x1_dc_motor}--\eqref{eq:x2_dc_motor} under the control law \eqref{eq:u}. 
    The first and second plots show the states $x_1$ and $x_2$, respectively, while the third plot shows the control input $u$.
    }
    \label{fig:closed_loop_motor}
\end{figure}

\begin{figure}[!ht]
    \centering
    \includegraphics[width=0.9\columnwidth]{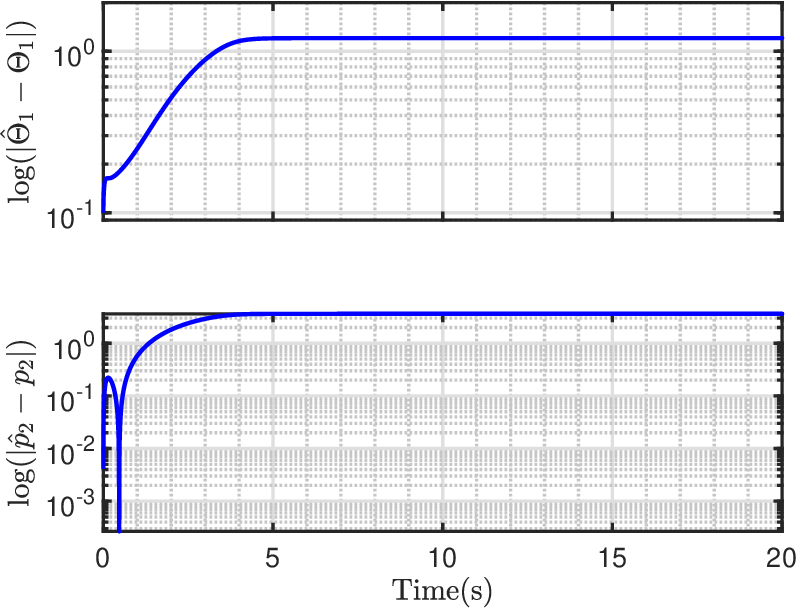}
    \caption{
    Estimation errors $|\hat{\Theta}_1 - \Theta_1|$ and $|\hat{p}_2 - p_2|$ shown on a logarithmic scale. 
    Consistent with the stability analysis, the parameter estimates do not converge to the true values $\Theta_1$ and $p_2$, but remain bounded.
    }
    \label{fig:estimations_motor}
\end{figure}

% \newpage
\section{Conclusions}
\label{sec:conclusions}

This paper developed an adaptive constraint-lifting control framework for a class of nonlinear systems with multiplicative uncertainties and unknown parameters. 
The proposed approach transforms the constrained control problem into an equivalent unconstrained representation, enabling systematic controller synthesis in the presence of parametric uncertainty. 
An adaptive control law was constructed to compensate for uncertainties in both the input gain and the system dynamics while ensuring forward invariance of a prescribed safe set. 
A Lyapunov-based analysis established boundedness of all closed-loop signals and asymptotic convergence of the system states. 
Numerical simulations demonstrated the effectiveness of the proposed method in achieving safe tracking under uncertainty.

% This paper extended a constraint-enforcing control framework to a class of nonlinear systems with multiplicative uncertainties and unknown parameters in the system dynamics.
% %
% The proposed approach preserves the transformation of the constrained control problem into an equivalent unconstrained one, enabling controller synthesis in the presence of parametric uncertainty.
% %
% An adaptive control law was developed to compensate for uncertainties in both the input gain and the dynamic map, while maintaining the positive invariance of a prescribed safe set.
% %
% A Lyapunov-based analysis establishes boundedness of all closed-loop signals and asymptotic convergence of the system states.
% %
% Numerical simulations demonstrate the effectiveness of the proposed approach in enforcing state constraints under uncertainty.

\appendices
\section{Dynamics in \texorpdfstring{$z$}{z} Coordinate}
\label{ap:derivation of z_dynamics}
This appendix derives the system dynamics in the lifted $z$-coordinates used for controller synthesis.

It follows from \eqref{eq:x1_to_z1} that 
\begin{align}
    \dot{z}_1
        &=
            \overline{x}_1
            \partial_{\chi_1} 
            \phi(\chi_1) 
            \dot \chi_1
        \nn \\ 
        &=
            \overline{x}_1
            \partial_{\chi_1} 
            \phi(\chi_1) 
            \overline{x}_1^{-1}
            \dot x_1
        \nn \\ 
        &=
            \partial_{\chi_1} 
            \phi(\chi_1)
            g_1(x_1)
            x_2
        \nn \\ 
        &=
            \partial_{\chi_1} 
            \phi(\psi_1(\zeta_1)) 
            g_1(\overline{x}_1\psi_1(\zeta_1))
            \overline{x}_2
            \psi_2(\zeta_2)
        \nn \\
        &=
            \Phi(\zeta_1) 
            \psi_2(\zeta_2)
        \nn \\
        &=
            \SG_1(z_1,z_2).
\end{align}
Similarly, it follows from \eqref{eq:x2_to_z2} that 
\begin{align}    
    \dot{z}_2
        &=
            \overline{x}_2
            \partial_{\chi_2} 
            \phi(\chi_2) 
            \dot \chi_2
        \nn \\ 
        &=
            \partial_{\chi_2} 
            \phi(\chi_2)
            \overline{x}_2
            \overline{x}_2^{-1}
            \dot x_2
        \nn \\ 
        &=
            \partial_{\chi_2} 
            \phi(\chi_2)
            \left(
                f_2(x_1, x_2) 
                + 
                g_2(x_1,x_2) u
            \right)
        \nn \\ 
        &=
            \partial_{\chi_2} 
            \phi(\psi_2(\zeta_2))
            \Big(
                f_2
                (
                    \overline{x}_1
                    \psi_1(\zeta_1),
                    \overline{x}_2
                    \psi_2(\zeta_2)
                ) 
        \nn\\
        & \quad
                + 
                g_2
                (
                    \overline{x}_2\psi_1(\zeta_1),
                    \overline{x}_2\psi_2(\zeta_2)
                ) 
                u 
            \Big)
        \nn \\
        &=
            \SF_2(z_1,z_2)  + \SG_2(z_1,z_2) u.
\end{align}

\section{Time-derivative of Sigmoid Integral-based Lyapunov Function}
\label{ap:appendix_sigmoid integral definition}
This appendix computes the time derivative of the sigmoid integral function used in the Lyapunov candidate, which is required for the stability analysis.

Note that 
\begin{align}
    \rmd_t 
    % \sum_{i=1}^n
            \SV(\zeta_{2})
                &=
                    % \sum_{i=1}^n
                   \rmd_t
                   % \int_0^{\zeta_{2i}}
                   \sigma(s)
                   \rmd s.
\end{align}
% For each $i = 1, \ldots, n,$
%
It follows from the Leibniz's formula that
\begin{align}
    \rmd_t
       % \int_0^{\zeta_{2i}}
       \sigma(s)
       \rmd s
        =
            \sigma(\zeta_{2})
            \dot{\zeta}_{2},
\end{align}
which, using \eqref{eq:zeta2_def}, implies that
\begin{align}
    \rmd_t 
    % \sum_{i=1}^n
            \SV(\zeta_{2})
                &=
                    % \sum_{i=1}^n
                   \sigma(\zeta_{2})
                   \dot{\zeta}_{2}
                =
                    \psi(\zeta)
                   \dot{\zeta}_{2}
                \nn \\
                &=
                    \psi(\zeta_2)
                    \overline{x}_2^{-1}
                    \dot{z}_2. 
\end{align}

\printbibliography
\end{document}